\documentclass[12pt, twoside]{article}
\usepackage{amsmath,amsthm,amssymb}
\usepackage{times}
\usepackage{enumerate}
\usepackage{url}
\usepackage{lineno}
\usepackage{microtype}
\usepackage[american]{babel}
\usepackage{indentfirst}
\usepackage{color}
\usepackage{cases}
\usepackage{geometry}
\usepackage{cite}
\def\titlerunning#1{\gdef\titrun{#1}}
\makeatletter
\def\author#1{\gdef\autrun{\def\and{\unskip, }#1}\gdef\@author{#1}}
\def\address#1{{\def\and{\\\hspace*{18pt}}\renewcommand{\thefootnote}{}%
\footnote {#1}}%
\markboth{\autrun}{\titrun}}
\makeatother
\def\email#1{e-mail: #1}
\def\subjclass#1{{\renewcommand{\thefootnote}{}%
\footnote{\emph{Mathematics Subject Classification (2020):} #1}}}
\def\keywords#1{\par\medskip
\noindent\textbf{Keywords.} #1}

\newtheorem{result}{\textbf{Theorem}}

\newtheorem{lemma}{Lemma}[section]
\newtheorem{definition}{Definition}[section]

\newtheorem{remark}{Remark}[section]
\newtheorem{example}{Example}[section]
\numberwithin{equation}{section}

\begin{document}
\baselineskip=15pt

\titlerunning{Time periodic solutions of first order mean field games}
\title{Time periodic solutions of first order mean field games from the perspective of Mather theory}
\author{Panrui Ni}

\maketitle
\address{Shanghai Center for Mathematical Sciences, Fudan University, Shanghai 200433, China; \\
\email{panruini@fudan.edu.cn}}


\vspace{-6ex}
\begin{abstract}
In this paper, the existence of non-trivial time periodic solutions of first order mean field games is proved. It is assumed that there is a non-trivial periodic orbit contained in the Mather set. The whole system is autonomous with a monotonic coupling term. Moreover, the large time convergence of solutions of first order mean field games to time periodic solutions is also considered.

\end{abstract}




\section{Introduction}

The mean field games (MFGs) was introduced by Lasry and Lions in \cite{Li1,Li2,Li3} and by Caines, Huang and Malham\'e in \cite{Hu1,Hu2} to analyze large population stochastic differential games. In this paper, we consider the evolutionary first order mean field games with a finite horizon $T>0$ and an initial-final condition, which is a system consists of a Hamilton-Jacobi equation coupled with a continuity equation (cf. \cite[Eqs. (49)-(51)]{Li3})
\begin{numcases}{}
  \partial_t u(x,t)+H(x,Du(x,t))=F(m(t))+c,\quad (x,t)\in M\times (0,T).\label{ME1}\\
  \partial_t m(x,t)+\textrm{div}(\partial_p H(x,D u(x,t))m(x,t))=0,\quad (x,t)\in M\times (0,T).\label{ME2}\\
  u(x,0)=\varphi(x),\quad m(x,T)=m_T(x).\label{ME3}
\end{numcases}
Here we assume $M$ is a compact connected and smooth Riemannian manifold without boundary. Let $D$ stand for the spacial gradient with respect to $x\in M$, and let $d(x,y)$ be the Riemannian distance between $x$ and $y$ in $M$. We denote by $TM$ and $T^*M$ the tangent and cotangent bundle over $M$ respectively. The function $H:T^*M\times\mathbb R\to\mathbb R$ is called the Hamiltonian. Let $\mathcal P(M)$ be the space of all Borel probability measures on $M$. The function $m(x,t)$ is the density of the mean field term $m(t)\in\mathcal P(M)$ if it is absolutely continuous with respect to the Lebesgue measure, i.e. $dm(t)=m(x,t)dx$. In previous works, the coupling term $F$ is usually assumed to depend on the spacial variable $x\in M$, i.e., $F=F(x,m)$ is from $M\times \mathcal P(M)$ to $\mathbb R$. In the present paper, we assume the coupling term $F:\mathcal P(M)\to \mathbb R$ is a functional independent of $x\in M$, and $c\in\mathbb R$ is a constant. Although $F$ does not depend on $x$, the results of this article can still handle the case where the coupling term is of the form $F(m)+f(x)$, because the Hamiltonian $H$ depends on $x$.

The solution of the above system is defined as follows.
\begin{definition}(Viscosity solution)
A continuous function $u:M\times[0,T]\rightarrow\mathbb R$ is called a viscosity subsolution (resp. supersolution) of (\ref{ME1}) if for each test function $\phi:M\times[0,T]\to\mathbb R$ of class $C^1$, when $u-\phi$ attains its local maximum (resp. minimum) at $(x,t)$, then
\begin{equation*}
  \partial_t \phi(x,t)+H(x,D\phi(x,t))\leq F(m(t))+c,\quad (\textrm{resp}.\ \geq F(m(t))+c).
\end{equation*}
A continuous function $u$ is called a viscosity solution of (\ref{ME1}) if it is both a viscosity subsolution and a viscosity supersolution.
\end{definition}
\begin{definition}\label{ws}
A weak solution of the system (\ref{ME1})-(\ref{ME2}) is a couple $(u(x,t),m(t))\in C(M\times[0,T])\times L^1([0,T],\mathcal P(M))$ such that (\ref{ME1}) is satisfied in viscosity sense and (\ref{ME2}) is satisfied in distributions sense, i.e.,
\[\int_{M}\varphi(x,T)dm_T(x)=\int_0^T\int_{M}(\partial_t\varphi(x,t)+\langle \partial_p H(x,D u(x,t)),D\varphi(x,t)\rangle) dm(t)(x),\]
for each $\varphi\in C^\infty_c(M\times(0,T])$. Here $\langle\cdot,\cdot\rangle$ represents the canonical pairing between the tangent space and cotangent space, 
$C^\infty_c(M\times(0,T])$ stands for the set of all smooth functions on $M\times(0,T]$ with compact support.
\end{definition}

Now we recall some facts on the space of probability measures $\mathcal P(M)$. A sequence $\{\mu_k\}_{k\in\mathbb N}\subset\mathcal P(M)$ is $w^*$-convergent to $\mu\in\mathcal P(M)$ if
\[\lim_{k\rightarrow+\infty}\int_M f(x)d\mu_k=\int_M f(x)d\mu,\quad \forall f\in C(M).\]
We shall work with the Monge-Wasserstein distance $d_1$, which is defined by
\[d_1(m_1,m_2)=\sup_\phi \int_{M} \phi d(m_1-m_2),\quad \forall m_1,\ m_2\in\mathcal P(M),\]
where the supremum is taken over all 1-Lipschitz continuous functions on $M$. We recall that $d_1$ metricizes the $w^*$-topology. Assume the coupling term $F:\mathcal P(M)\to\mathbb R$ satisfies
\begin{itemize}
	\item[{\bf (F)}] {\it Lipschitz continuity:} there is $C>0$ such that $|F(m_1)-F(m_2)|\leq Cd_1(m_1,m_2)$.
\end{itemize}

Consider the stationary system corresponding to (\ref{ME1})-(\ref{ME3})
\begin{numcases}{}
  H(x,Du(x))=F(m)+c,\quad x\in M.\label{ME01}\\
  \textrm{div}(\partial_p H(x,Du(x))m(x))=0,\quad x \in M.\label{ME02}
\end{numcases}
A natural goal of the study of MFGs is to prove the convergence of solutions of evolutionary MFGs (\ref{ME1})-(\ref{ME3}) to equilibrium states (\ref{ME01})-(\ref{ME02}) as the horizon $T\to+\infty$. 
There are many works on the long time average of solutions of evolutionary MFGs. 
Usually, we assume $F:M\times\mathcal P(M)\to\mathbb R$ satisfies the monotone condition
\begin{equation}\label{monoc}
  \int_M (F(x,m)-F(x,m'))d(m-m')\geq 0, \quad \forall m,\ m'\in\mathcal P(M).
\end{equation}
Without the monotone condition, the solution of evolutionary MFGs may not converges to equilibrium states, see \cite{tp00,tp1,tp3,tp2}. In this paper, we provide a case in which there are non-trivial time periodic solutions of first order MFGs. Since $F:\mathbb P(M)\to\mathbb R$ is independent of $x\in M$, for all probability measures $m$ and $m'$, we have \[\int_M (F(m)-F(m'))d(m-m')=(F(m)-F(m'))\int_Md(m-m')=0.\] Then the coupling term $F:\mathcal P(M)\to\mathbb R$ considered in the present paper satisfies the monotone condition (\ref{monoc}). The results in this paper show that the large time behavior of MFGs is related to the minimal measures introduced by J. Mather in \cite{Mat}. For the further study, we may consider the coupling term $F$ depending on the spacial variable $x\in M$.

Before presenting the results obtained in this paper, we first recall some recent works on the large time behavior of mean field games. For the second order mean field games, the large time behavior is quite well understood. One can refer to \cite{Car2,Mas,Car3,Car30,Cir}. For the first order mean field games, the analysis is related to the weak KAM theory and the Aubry-Mather theory. There are many works on the long time average, see \cite{Car1,Can2,Can1,Car4,Carf}. In these works, the monotone condition (\ref{monoc}) plays a center role. In \cite{Bar}, the author proved the uniform convergence of $u$ and the convergence of $m$ in the distance $d_1$ to stationary solutions when the Hamiltonian of the form $|p|^2/2$. For mean field games with non-monotone coupling term, there are also some works on the existence of time periodic solutions. For congestion problem, \cite{tp2} discussed the traveling wave solution, which is the first example of time periodic solutions. For second order mean field games with the Hamiltonian of the form $|p|^2/2$, \cite{tp1} proved the existence of time periodic solutions, and \cite{tp3} discussed the existence of oscillating solutions when there are two populations of agents. In \cite{tp00}, by discussing the minimal points of an energy functional, the authors found time periodic solutions of first order mean field games on $\mathbb R^n$.

In \cite{Mat}, in order to generalize the Aubry-Mather theory for twist maps, J. Mather introduced the concept of minimal measures, which are also called the Mather measures. Assume the Hamiltonian $H:T^*M\to\mathbb R$ is $C^2$, and satisfies the Tonelli condition
\begin{itemize}
\item[{\bf (H1)}] {\it Strictly convexity:} the Hessian matrix $\partial^2_{pp}H(x,p)$ is positively definite for all $(x,p)\in T^*M$.
\item[{\bf (H2)}] {\it Superlinearity:} $H(x,p)/|p|\to+\infty$ as $|p|\to+\infty$ for all $x\in M$.
\end{itemize}
Define the Legendre dual of $H$, which is called the Lagrangian as follows
\begin{equation*}
	L(x,\dot{x}):=\sup_{p\in T^*_xM}\{\langle\dot{x},p\rangle-H(x,p)\}.
\end{equation*}
Then $L:TM\to\mathbb R$ is also of class $C^2$, and satisfies the Tonelli condition with respect to $\dot x$ (cf. \cite[Appendix]{CS}). For each $x,y\in M$, consider the action minimizing problem
\begin{equation}\label{ht}
  h_t(x,y):=\min \int_0^t L(\gamma(s),\dot \gamma(s))ds,
\end{equation}
where the minimum is taken among absolutely continuous curves $\gamma:[0,t]\rightarrow M$ with $\gamma(0)=x$ and $\gamma(t)=y$. Here $h_t$ is called the minimal action. By the Tonelli Theorem, the minimum can be achieved, see \cite{Fat-b}. The curves achieving the minimum satisfy the Euler-Lagrange equation
\[\frac{d}{dt}\bigg(\frac{\partial L}{\partial \dot x}(x,\dot x)\bigg)=\frac{\partial L}{\partial x}(x,\dot x).\]
Denote by $\Phi^L_t$ the flow of the Euler-Lagrange equation. Denote by $\Phi^H_t$ the dual flow of $\Phi^L_t$ through the Legendre transformation, which is the flow of the Hamilton equation
\begin{equation}\label{Heq}
  \dot x=\frac{\partial H}{\partial p}(x,p),\quad \dot p=-\frac{\partial H}{\partial x}(x,p).
\end{equation}
In the following, we denote by $m$ the measures on $M$ and $\mu$ the measures on $TM$. Consider the action minimizing problem
\[-c_0=\min_{\mu} \int_{TM}L(x,\dot x)d\mu,\]
where $\mu$ varies among all Borel probability measures on $TM$ invariant under the action of the flow $\Phi^L_t$. The constant $c_0$ is the unique constant $c\in\mathbb R$ such that $H(x,Du)=c$ has solutions, and is called the critical value or the effective Hamiltonian, cf. \cite[Theorem 1]{hom}. We call the measures achieve this minimum the minimal measures, and denote by $\tilde{\mathfrak M}$ the set of all minimal measures. Define the set of all projected minimal measures as
\[\mathfrak M:=\{m\in\mathcal P(M):\ m=\pi_\# \mu,\ \mu\in\tilde{\mathfrak M}\},\]
where $\pi:TM\to M$ is the canonical projection, and $\pi_\#$ is the push-forward induced by $\pi$. The Mather set is defined as
\[\tilde{\mathcal M}=\overline{\bigcup_{\mu}\ \textrm{supp}(\mu)}\subset TM,\]
where $\textrm{supp}(\mu)$ is the support of $\mu$, $\bar A$ is the closure of $A\subset TM$, and the union is taken over all minimal measures $\mu\in\tilde{\mathfrak M}$. Let $\mathcal M=\pi(\tilde{\mathcal M})$ be the projected Mather set. By \cite{Mat}, $\tilde{\mathcal M}$ is a Lipschitz graph over $\mathcal M$.

Now we show that Mather's minimal measures give solutions of (\ref{ME01})-(\ref{ME02}). Let $\mu\in\tilde{\mathfrak M}$. Since $\mu$ is $\Phi^L_t$-invariant, by \cite{Mat}, we have
\[\int_{TM}\langle \dot x,d\varphi(x)\rangle d\mu=0,\]
By \cite{Fat-b}, let $u_0$ be a solution of
\begin{equation}\label{H=c0}
  H(x,Du)=c_0.
\end{equation}
Then $u_0$ is differentiable on $\mathcal M$, cf. \cite[Theorem 5.2.2]{Fat-b}. The derivative of $u_0$ is given by
\begin{equation}\label{Du-}
  Du_0(x)=\frac{\partial L}{\partial \dot x}(x,\dot x),\quad (x,\dot x)\in\tilde{\mathcal M},
\end{equation}
which is independent of $u_0$, and is Lipschitz continuous. Let $m_0=\pi_{\#}\mu\in\mathfrak M$ be a projected minimal measure. Then
\[\int_M \bigg\langle \frac{\partial H}{\partial p}(x,Du_0(x)),d\varphi(x)\bigg\rangle dm_0=\int_{\tilde{\mathcal M}}\langle \dot x,d\varphi(x)\rangle d\mu=0,\]
which implies that $(u_0,m_0)$ is a solution of (\ref{ME01})-(\ref{ME02}) with $c=c_0-F(m_0)$. Assume
\begin{itemize}
\item[{\bf (H3)}] there is a non-trivial periodic orbit $\Gamma(t)=(\gamma(t),\dot\gamma(t))$ on $\tilde{\mathcal M}$ with period $\tau$.
\end{itemize}
Here we recall that, \cite{Man} proved that the minimal measure is unique for generic Lagrangian. It was conjectured by R. Ma\~n\'e that, the unique minimal measure is supported on a fixed point or a periodic orbit.

\begin{remark}
For the one dimensional case $M=\mathbb T\simeq \mathbb R/\mathbb Z$, we can give more description on (H3) from the perspective of the Mather's theory. Let $\alpha:\mathbb R\to\mathbb R$ be the Mather's $\alpha$-function on the cohomology (see \cite[Section 4.10]{Fat-b}). Then for each $a\in\mathbb R$, $\alpha(a)$ is the unique constant $c$ such that $H(x,a+Du)=c$ is solvable. By Lemma \ref{uC2} below, the Mather set consists of either fixed points or exactly one periodic orbit. Assume that there is a fixed point in the Mather set $\tilde{\mathcal M}$. Then the rotation number of all minimal measures is equal to zero, Equivalently, $\alpha'(a)=0$, one can see \cite{Mat} for more details. If (H3) holds, then $\alpha'(a)\neq 0$. Consider
\[H_a(x,p)=\frac{1}{2}(p+a)^2+V(x),\]
where $V:\mathbb T\to\mathbb R$ that attains its maximum at exactly two points $0$ and $X$ and such that $V(0)=V(X)=0$. Let $a_0:=\int_0^1\sqrt{-2V(x)}dx$. Then $\alpha(a)=0$ for $a\in[-a_0,a_0]$ and $\alpha(a)>0$ for $a\notin [-a_0,a_0]$ (see \cite[Chapter 4]{Za2} for details). Since $\alpha:\mathbb R\to\mathbb R$ is convex (see \cite{Mat}), $\alpha'(a)\neq 0$ holds for $a\notin [-a_0,a_0]$, which implies that (H3) holds.
\end{remark}

Define
\[\mathfrak N:=\{m\notin\mathfrak M,\ \textrm{supp}(m)\subset\{\gamma\}\}\subset\mathcal P(M).\]

\begin{result}\label{m1}
Assume (H1)-(H3)(F) and $m_T\in\mathfrak N$, then there is $c(m_T)\in\mathbb R$ such that the system (\ref{ME1})-(\ref{ME3}) has a non-trivial time periodic solution with period $\tau$ when $c=c(m_T)$.
\end{result}

\begin{result}\label{m2}
Assume $M=\mathbb T$, (H1)-(H3)(F). Given $m_T\in\mathcal P(M)\backslash\mathfrak M$, then there is $c(m_T)\in\mathbb R$ such that the system (\ref{ME1})-(\ref{ME3}) has a non-trivial time periodic solution $(\bar u(x,t),\bar m(t))$ when $c=c(m_T)$. If $m_T$ is an absolutely continuous measure with respect to the Lebesgue measure on $\mathbb T$, then it has the density $m_T(x)$, and we have
\begin{itemize}
\item [(1)] 
$\bar u(x,t)\in C^1(\mathbb T\times[0,T])$ is unique up to a constant, $\bar m(x,t)$ is unique, and $c(m_T)$ is the unique constant $c\in\mathbb R$ such that (\ref{ME1})-(\ref{ME3}) has time periodic solutions.
\item [(2)] The map $m_T\mapsto c(m_T)$ is Lipschitz continuous in $d_1$.
\item [(3)] If $m_T(x)\in C^1(\mathbb T)$, then $\bar m(x,t)\in C^1(\mathbb T\times[0,T])$.
\item [(4)] For each $\varphi\in C(\mathbb T)$ and $c\in\mathbb R$, there is a unique weak solution $(u(x,t),m(t))\in C(\mathbb T\times[0,T])\times Lip([0,T],\mathcal P(\mathbb T))$ of (\ref{ME1})-(\ref{ME3}). Let $c=c(m_T)$. If $m_T(x)$ is bounded, then for each $t\in \mathbb [0,+\infty)$, we have
\[u(x,s)-\int_0^{T-t} F(m(\tau))d\tau\to\ \bar u(x,s)-\int_0^{T-t} F(\bar m(\tau))d\tau\quad \textrm{as}\quad T\to+\infty\]
uniformly in $(x,s)\in \mathbb T\times[T-t,T]$ and
\[d_1(m(s),\bar m(s))\to 0\quad \textrm{as}\quad T\to+\infty\]
uniformly in $s\in [T-t,T]$. Here $(\bar u(x,t), \bar m(t))$ is a time periodic solution.
\end{itemize}
\end{result}

\begin{remark}
The main differences between the present work and \cite{tp00} are
\begin{itemize}
\item The method in the present paper is from the weak KAM theory and the Aubry-Mather theory instead of finding minimal points of an energy functional.

\item In \cite{tp00}, the potential of the coupling term is a sort of ``double-well" potential. Then the convexity of the potential does not hold, which implies that the coupling term is not monotone. In the present paper, as stated before, the monotone condition (\ref{monoc}) holds.

\item The time periodic solutions given in the present paper are weak solutions defined in Definition \ref{ws}, which may have more regularity. In particular, for $M=\mathbb T$, the time periodic solutions can be classical solutions of class $C^1$, see Items (1)(3) in Theorem \ref{m2}. 
\item Item (4) in Theorem \ref{m2} considers the uniform convergence of $u$ and the convergence of $m$ in $d_1$ as the horizon $T\to+\infty$.
\end{itemize}
\end{remark}
At last, we provide a specific example of the non-trivial smooth time periodic solution of first order mean field games.
\begin{example}
Consider $x\in\mathbb T^n\simeq [0,1]^n$ and
\begin{numcases}{}
  \partial_t u(x,t)+\sum_{i=1}^n\bigg(\frac{1}{2}(\partial_{x_i} u(x,t))^2-\partial_{x_i} u(x,t)\bigg)=F(m(t)),\label{E1}\\
  \partial_t m(x,t)+\sum_{i=1}^n\partial_{x_i}\big((\partial_{x_i} u(x,t)-1)m(x,t)\big)=0.\label{E2}
\end{numcases}
where
\[F(m)=\int_{[0,1]^n} 4\pi\cos\big(2\pi \sum_{i=1}^nx_i\big)m(x)dx.\]
It is direct to check that
\begin{equation*}
  \left\{
   \begin{aligned}
   &\bar u(x,t)=\sin(2\pi t),
   \\
   &\bar m(x,t)=1+\cos\big(2\pi (\sum_{i=1}^nx_i+t)\big).
   \\
   \end{aligned}
   \right.
\end{equation*}
is a non-trivial time periodic solution of the above mean field games model (\ref{E1})-(\ref{E2}). At the same time, there is a stationary solution $(u_0,m_0)=(0,1)$.
\end{example}

\section{Proof of Theorem \ref{m1}}\label{p1}


Let $\Gamma(t)=(\gamma(t),\dot \gamma(t))$ with $t\in\mathbb R$ be a periodic orbit on $\tilde{\mathcal M}$ with period $\tau$. Consider $m_T\in\mathfrak N$, that is, $m_T\notin\mathfrak M$ supported on $\{\gamma\}\subseteq\mathcal M$. For $x=\gamma(t_0)$, define
\[\Phi(t,T,x)=\gamma(t+t_0-T).\]
Let $u_0(x)$ be a viscosity solution of (\ref{H=c0}). Define
\begin{equation}\label{sol}
  \left\{
   \begin{aligned}
   &\bar u(x,t)=u_0(x)+\int_0^tF(\bar m(s))ds-\frac{t}{\tau}\int_0^\tau F(\bar m(s))ds,
   \\
   &\bar m(t)=\Phi(t,T,\cdot)_{\#}m_T.
   \\
   \end{aligned}
   \right.
\end{equation}
\begin{lemma}
$\bar m(t)$ is periodic with period $\tau$.
\end{lemma}
\begin{proof}
By the definition of push-forward, it is direct to see that for all $f\in C(M)$, we have
\[\int_M f(x)d\bar m(t)=\int_{\gamma}f(x)d\bar m(t)=\int_{\gamma}f(\Phi(t,T,x))dm_T,\]
and
\begin{align*}
\int_M f(x)d\bar m(t+\tau)&=\int_{\gamma}f(\Phi(t+\tau,T,x))dm_T
\\ &=\int_{\gamma}f(\Phi(t,T,x))dm_T=\int_M f(x)d\bar m(t),
\end{align*}
which implies that $t\mapsto \bar m(t)$ is periodic with period $\tau$.
\end{proof}

Note that $u_0$ is differentiable on $\gamma$, where $Du_0(x)$ is Lipschitz continuous and is given by (\ref{Du-}). Since $m_T$ is supported on $\gamma$ and $\gamma$ is invariant under $\Phi(t,T,\cdot)$, similar to \cite[Section 4.2]{Car5}, we can prove that $\bar m(t)$ is the unique weak solution of
\begin{equation*}\label{e2}
  \partial_t m(x,t)+\textrm{div}(\partial_pH(x,Du_0(x))m(x,t))=0,\quad m(x,T)=m_T(x).
\end{equation*}
Also, one can check that $\bar u(x,t)$ is the unique solution of
\[\partial_t u(x,t)+H(x,Du(x,t))=F(\bar m(t))+c_0-\frac{1}{\tau}\int_0^\tau F(\bar m(s))ds,\quad u(x,0)=u_-(x).\]
Note that $Du(x,t)=Du_0(x)$, we get a time periodic solution $(\bar u(x,t),\bar m(x,t))$ of (\ref{ME1})-(\ref{ME3}) with period $\tau$, where
\begin{equation}\label{cmt}
  c(m_T)=c_0-\frac{1}{\tau}\int_0^\tau F(\bar m(s))ds.
\end{equation}
\begin{lemma}
The periodic solution $\bar m(t)$ is non-trivial.
\end{lemma}
\begin{proof}
First we recall that $\tilde{\mathcal M}$ is a Lipschitz graph over $\mathcal M$, then the lift $\mu_T$ of $m_T$ to $\tilde{\mathcal M}$ is unique. Since $m_T\in \mathfrak N$, $\mu_T\notin \tilde{\mathfrak M}$. According to \cite{Man2}, all invariant measures supported on $\tilde{\mathcal M}$ are minimal measures, $\mu_T$ is not invariant under $\Phi^L_t$. Since $\textrm{supp}(\mu_T)\subset \Gamma$, $\Phi^L_{t-T\#}\mu_T$ is periodic in $t$ with period $\tau$. For every $(x,\dot x)=\Gamma(t_0)\in \textrm{supp}(\mu_T)$, we have
\[\Phi^L_{t-T}(\Gamma(t_0))=\Gamma(t+t_0-T).\]
Let $y=\Phi(t,T,x)$, define the inverse $x=\Phi^{-1}(t,T,y)$.  Then we have
\[\pi^{-1}\Phi^{-1}(t,T,x)=\Gamma(t_0-(t-T))=(\Phi^L_{t-T})^{-1}\pi^{-1}(x),\quad \forall x\in\{\gamma\}.\]
Thus, for every $A\subseteq \{\gamma\}$, we have

\begin{equation}\label{111}
\begin{aligned}
\pi_{\#}(\Phi^L_{t-T\#}\mu_T)(A)&=\mu_T((\Phi^L_{t-T})^{-1}(\pi^{-1}(A)))
\\ &=\mu_T(\pi^{-1}(\Phi^{-1}(t,T,A)))=\Phi(t,T,\cdot)_{\#}(\pi_{\#}\mu_T)(A).
\end{aligned}
\end{equation}
Since $\pi^{-1}$ is a bi-Lipschitz continuous map on $\gamma$ by \cite{Mat}, and $\mu_T$ is not invariant under $\Phi^L_t$, there is $A\subseteq \{\gamma\}$ and $t_1\in (T-\tau,T)$ such that
\[\Phi^L_{t_1-T\#}\mu_T(\pi^{-1}(A))\neq \mu_T(\pi^{-1}(A)).\]
Therefore, by (\ref{111}) we get
\begin{align*}
m_T(A)&=\pi_{\#}\mu_T(A)=\mu_T(\pi^{-1}(A))
\\ &\neq \Phi^L_{t_1-T\#}\mu_T(\pi^{-1}(A))=\pi_{\#}(\Phi^L_{t_1-T\#}\mu_T)(A)
=\Phi(t_1,T,\cdot)_{\#}(\pi_{\#}\mu_T)(A)
\\&=\Phi(t_1,T,\cdot)_{\#}m_T(A)=\bar m(t_1)(A).
\end{align*}
Thus, $\bar m(t)$ is non-trivial.
\end{proof}


\section{Basic properties of (\ref{ME1})-(\ref{ME3})}

According to \cite{Fat-b}, the viscosity solution of
\begin{equation*}
  \left\{
   \begin{aligned}
   &\partial_t w(x,t)+H(x,D w(x,t))=0,\quad (x,t)\in M\times(0,+\infty).
   \\
   &u(x,0)=\varphi(x),\quad x\in M.
   \\
   \end{aligned}
   \right.
\end{equation*}
has the following representation formula called the Lax-Oleinik formula
\begin{equation}\label{LO}
  w(x,t)=\min_{z\in M}\bigg\{\varphi(z)+h_t(z,x)\bigg\}=\min_{\gamma}\bigg\{\varphi(\gamma(0))+\int_0^t L(\gamma(s),\dot\gamma(s))ds\bigg\},
\end{equation}
where $h_t$ is defined by (\ref{ht}) and $\gamma:[0,t]\to M$ is taken among absolutely continuous curves with $\gamma(t)=x$. By the Tonelli Theorem, the minimum in the definition of $w(x,t)$ can be achieved, and we call the curve achieving the minimum a minimizer. We have known that $w(x,t)$ is uniformly bounded and equi-Lipschitz continuous for all $t>1$. Moreover, $w(x,t)+c_0t$ uniformly converges to a solution of (\ref{H=c0}) as $t\to+\infty$. For these results, one can refer to \cite[Proposition 4.6.6 and Theorem 6.3.1]{Fat-b}. 

Similar to \cite[Section 4.1]{Car5}, for each $(x,t)\in M\times[0,T]$, there is a measurable selection $\gamma(\cdot;x,t):[0,t]\to M$ of minimizers of $w(x,t)$. Then we can define the optimal flow as
\[\Psi(s,t,x)=\gamma(s;x,t),\quad s\in[0,t].\]
Similar to the proof of \cite[Theorem 4.18]{Car5}, we can prove that
\begin{lemma}\label{unim}
If $m_T\in\mathcal P(M)$ is absolutely continuous with respect to the Lebesgue measure, $m(t)=\Psi(t,T,\cdot)_{\#}m_T$ is the unique weak solution of
\begin{equation}\label{divm}
  \partial_t m(x,t)+\textrm{div}(\partial_p H(x,Dw(x,t))m(x,t))=0,\quad m(T)=m_T.
\end{equation}
\end{lemma}
\noindent \textit{Outline of the proof.} By \cite[Theorem 6.4.7 and Theorem 6.4.8]{CS}, $Dw(\Psi(s,t,x),s)$ exists for $s\in(0,t)$, and
\[Dw(\Psi(s,t,x),s)=\partial_{\dot x} L(\Psi(s,t,x),\partial_s\Psi(s,t,x)).\]
Similar to \cite[Lemma 4.15]{Car5}, $m(t)$ is a weak solution of (\ref{divm}). It remains to prove the uniqueness of the weak solution of (\ref{divm}). The proof is based on the fact that, if $w(x,T)$ is differentiable at $x$, then the minimizer $\gamma(\cdot;x,T):[0,T]\to M$ of $w(x,T)$ is unique (cf. \cite{CS}). Therefore, the optimal flow $\Psi(t,T,x)$ is uniquely given by $\gamma(t;x,T)$. Since $x\mapsto w(x,T)$ is Lipschitz continuous, the optimal flow $\Psi(t,T,x)$ is uniquely defined for almost all $x\in M$. Since $m_T$ is absolutely continuous with respect to the Lebesgue measure, the optimal flow $\Psi(t,T,x)$ is uniquely defined for $m_T$-almost all $x\in M$. Thus, the weak solution of (\ref{divm}) is uniquely given by $\Psi(t,T,\cdot)_{\#}m_T$.\qed

\begin{lemma}\label{unique}
For each $c\in\mathbb R$, $\varphi(x)\in C(M)$, and $m_T\in\mathcal P(M)$ which is absolutely continuous with respect to the Lebesgue measure, there is a unique weak solution $(u(x,t),m(x,t))\in C(M\times[0,T])\times Lip([0,T],\mathcal P(M))$ of (\ref{ME1})-(\ref{ME3}).
\end{lemma}
\begin{proof}
For all $m\in\mathcal P(M)$, the operator norm of $m$ is bounded. Then $\mathcal P(M)$ is compact in the $w^*$-topology. By the $w^*$-compactness of $\mathcal P(M)$, $F(m)$ is bounded. For each measurable map $t\mapsto \hat m(t)$, $F(\hat m(s))\in L^1([0,T])$. Then according to \cite[Theorem 5.1]{Intr},
\[u(x,t;\hat m)=w(x,t)+\int_0^t F(\hat m(s))ds+ct\]
is the unique solution of
\[\partial_t u(x,t)+H(x,Du(x,t))=F(\hat m(t))+c,\quad u(x,0)=\varphi(x).\]
To see the uniqueness of solution of (\ref{ME1})-(\ref{ME3}), let $(u(x,t),m(x,t))$ be a solution. Note that $Du(x,t;\hat m)=Dw(x,t)$ for all $\hat m$, by Lemma \ref{unim}, $m(t)=\Psi(t,T,\cdot)_{\#}m_T$ is unique. Then the unique solution of (\ref{ME1})-(\ref{ME3}) is given by
\begin{equation}\label{sole}
  \left\{
   \begin{aligned}
   &u(x,t)=w(x,t)+\int_0^tF(m(s))ds+ct,
   \\
   &m(t)=\Psi(t,T,\cdot)_{\#}m_T.
   \\
   \end{aligned}
   \right.
\end{equation}
It remains to prove that $t\mapsto m(t)$ is Lipschitz continuous. We have
\begin{align*}
d_1(m(t),m(s))&=\sup_{\phi}\int_{M}\phi(x)d(m(t)-m(s))
\\ &=\sup_{\phi}\int_{M}(\phi(\Psi(t,T,x))-\phi(\Psi(s,T,x)))dm_T
\\ &\leq \int_Md(\Psi(t,T,x),\Psi(s,T,x))dm_T.
\end{align*}
By \cite[Theorem 6.3.3]{CS}, $(\Psi(t,T,x),p(t))$ satisfies (\ref{Heq}), where the dual arc is defined by
\[p(t)=\frac{\partial L}{\partial \dot x}(\Psi(t,T,x),\partial_t\Psi(t,T,x)).\]
For almost all $x\in M$ and $t\in[0,T]$, by \cite[Theorem 6.4.9]{CS} we have
\[H(\Psi(t,T,x),p(t))=H(x,Dw(x,T)).\]
Since $\|Dw(x,T)\|_\infty$ is bounded for each given $T>0$, $p(t)$ is uniformly bounded by (H2). Then
\[\partial_t\Psi(t,T,x)=\frac{\partial H}{\partial p}(\Psi(t,T,x),p(t))\]
is also bounded uniformly in $x\in M$. Therefore, \[d_1(m(t),m(s))\leq \|\partial_t\Psi(t,T,x)\|_\infty|t-s|,\] which implies that $m(t)\in Lip([0,T],\mathcal P(M))$.
\end{proof}

\begin{lemma}\label{equi}
The family $\{w(x,t)\}_{t\geq 1}$ is equi-semi-concave. Moreover, let $u_0:=\lim_{t\to+\infty}w(x,t)$. Then $Dw(x,t)\to Du_0(x)$ for almost every $x\in M$.
\end{lemma}
\begin{proof}
By \cite[Theorem B.7]{JAMS}, the map $(y,x)\mapsto h_t(y,x)$ is equi-semi-concave for all $t\geq 1$. Since the argument is local, without any loss of generality, we assume that $M$ is an open subset $U$ of $\mathbb R^n$. Let $x$, $y\in U$, $\lambda\in[0,1]$. We denote by $z_\lambda$ the minimal point in (\ref{LO}) with the value function equals $w(\lambda x+(1-\lambda)y,t)$. Then we have
\begin{align*}
&\lambda w(x,t)+(1-\lambda)w(y,t)-w(\lambda x+(1-\lambda)y,t)
\\ \leq& \lambda h_t(z_\lambda,x)+(1-\lambda)h_t(z_\lambda,y)-h_t(z_\lambda,\lambda x+(1-\lambda)y)
\\ \leq& C\lambda(1-\lambda)|x-y|^2,
\end{align*}
where $C$ is the semi-concave constant of $\{h_t\}_{t\geq 1}$. Therefore, $\{w(x,t)\}_{t\geq 1}$ is equi-semi-concave. The last statement holds according to \cite[Theorem 3.3.3]{CS} or \cite[Lemma 4.6]{Car5}.
\end{proof}

\section{Proof of Theorem \ref{m2}}


We first consider the uniqueness and smoothness of the solution of (\ref{H=c0}). 
\begin{lemma}\label{uC2}
Let (H1)(H2) hold, and $M=\mathbb T$. The Mather set $\tilde{\mathcal M}$ of (\ref{H=c0}) consists of either fixed points or a periodic orbit $\Gamma=(\gamma,\dot\gamma)$. For the last case, we have
\begin{itemize}
\item[(1)] $\mathcal M=\{\gamma\}=\mathbb T$. Let $\bar\gamma:\mathbb R\to\mathbb R$ be the lift of $\gamma:\mathbb R\to\mathbb T$ on the cover space $\mathbb R$ over $\mathbb T$. Then $\bar\gamma:\mathbb R\to\mathbb R$ is monotone and $\dot\gamma\neq 0$.
\item[(2)] The viscosity solution $u_0$ of (\ref{H=c0}) is of class $C^2$ and unique up to a constant.
\end{itemize}
\end{lemma}
\begin{proof}
Since $M=\mathbb T$, by the Poincar\'e-Bendixon Theorem for planar dynamical systems and the recurrence property of the Mather set, $\tilde{\mathcal M}$ consists of fixed points or periodic orbits. Let $\Gamma(t)=(\gamma(t),\dot\gamma(t))$ be a time periodic orbit contained in $\tilde{\mathcal M}$. Assume $t\mapsto \bar\gamma(t)$ is not monotone. Then there are two times $t_1<t_2$, such that $\gamma(t_1)=\gamma(t_2)$ and $\dot\gamma_-(t_1)\neq \dot\gamma_-(t_2)$, which contradicts the graph property. Therefore, $t\mapsto \bar\gamma(t)$ is monotone. Since $\gamma$ is periodic, $\mathcal M=\{\gamma\}=\mathbb T^1$.

Let $u_0$ be a solution of (\ref{H=c0}). By $\mathcal M=\mathbb T$ and (\ref{Du-}), $u_0$ is differentiable on $\mathbb T$ with the same derivative. By the Lipschitz graph property of $\tilde{\mathcal M}$, $Du_0(x)$ is Lipschitz continuous. Therefore, $u_0$ is of class $C^{1,1}$. For each $x_0\in\mathbb T$ and each solution $v_0$ of (\ref{H=c0}), we have
\[v_0(x)=v_0(x_0)+\int_{x_0}^xDv_0(x)dx=v_0(x_0)+\int_{x_0}^xDu_0(x)dx.\]
Then $u_0$ is unique up to a constant.

We then show that $\dot\gamma\neq 0$. Since $p(x):=Du_0(x)$ is Lipschitz continuous, $Dp(x)$ exists almost everywhere. Here we recall the reachable gradient of the Lipschitz continuous function $p:\mathbb T\to\mathbb R$ (see \cite{CS})
\[D^*p(x):=\{q:\ \exists x_n\to x,\ Dp(x_n)\ \textrm{exists\ and}\ Dp(x_n)\to q\},\]
which is compact. Assume there is $t_0$ such that $\dot\gamma(t_0)=0$. Take $x_n\to \gamma(t_0)$ with $Dp(x_n)$ exists. Since $t\mapsto \bar\gamma(t)$ is monotone and $\{\gamma\}=\mathbb T$, there is $t_n\to t_0$ with $\gamma(t_n)=x_n$. Then we have
\[\dot p(\gamma(t_n))=\frac{d}{dt}\bigg|_{t=t_n}Du_0(\gamma(t))=\frac{d}{dt}\bigg|_{t=t_n}p(\gamma(t))=Dp(\gamma(t_n))\cdot \dot \gamma(t_n).\]
Define $p(t)=Du_0(\gamma(t))$, then $(\gamma(t),p(t))$ satisfies the Hamilton equation (\ref{Heq}). Note that $t\mapsto (\gamma(t),p(t))$ is $C^1$, then $t\mapsto (\dot\gamma(t),\dot p(t))$ is continuous. Let $t_n\to t_0$. By continuity, we conclude that $\dot p(t_0)=0$. Then $(\gamma(t_0),p(t_0))$ is a fixed point on $\tilde{\mathcal M}$, which leads to a contradiction.

It remains to show that $u_0$ is $C^2$, that is, $Du_0(x)$ is $C^1$. Since $\dot\gamma\neq 0$, there is a inverse for $t\mapsto\gamma(t)$, which is denoted by $t(x)$. Let
$p(t)=\partial_{\dot x}L(\gamma(t),\dot\gamma(t))$ and $p(x)=p(t(x))$. Since $\dot\gamma\neq 0$, we have
\begin{equation}\label{neq0}
  \dot\gamma=\frac{\partial H}{\partial p}(\gamma(t),p(t))\neq 0.
\end{equation}
Then by (\ref{Heq}), $(t(x),p(x))$ satisfies
\[\frac{dt}{dx}=(\partial_pH(x,p(x)))^{-1},\quad \frac{dp}{dx}=-(\partial_pH(x,p(x)))^{-1}\partial_xH(x,p(x)).\]
Since $H$ is $C^2$, $(t(x),p(x))$ is $C^1$. By (\ref{Du-}), $p(x)=Du_0(x)$ is $C^1$.
\end{proof}

Since $\{\gamma\}=\mathbb T$ by Lemma \ref{uC2}, $m_T\in\mathfrak N$ in the present case. The existence of non-trivial time periodic solutions of (\ref{ME1})-(\ref{ME3}) is given in Theorem \ref{m1}.

\subsection{Proof of Item (1)}

Now we prove the uniqueness of the time periodic solution. Let $(\hat u(x,t),\hat m(x,t))$ be a time periodic solution of (\ref{ME1})-(\ref{ME3}) with period $\bar\tau$, where the constant $c\in\mathbb R$ may not be equal to $c(m_T)$. By (\ref{sole}), Lemma \ref{equi} and the periodicity of $\hat u(x,t)$, for $n\in\mathbb N$, we have
\[D\varphi(x)=\lim_{n\to+\infty}D\hat u(x,n\bar\tau)=\lim_{n\to+\infty}Dw(x,n\bar\tau)=Du_0(x)\]
for almost every $x\in\mathbb T$. Therefore, $\varphi=u_0$, which implies that
\[\hat u(x,t)=u_0(x)+\int_0^tF(m(s))ds+(c-c_0)t.\]
Note that $D\hat u(x,t)=Du_0(x)$. By Lemma \ref{unique}, $\hat m(t)$ is uniquely given by $\Phi(t,T,\cdot)_{\#}m_T$, which equals to $\bar m(t)$. Then
\[c=c_0-\frac{1}{\tau}\int_0^\tau F(\bar m(s))ds\]
is the unique constant such that $\hat u(x,t)$ is bounded. We finally conclude that $(\hat u(x,t),\hat m(x,t))$ is the time periodic solution given in (\ref{sol}). Since $t\mapsto F(m(t))$ is continuous, $\hat u(x,t)\in C^1(\mathbb T^1\times[0,T])$. By Lemma \ref{uC2}, $\hat u(x,t)$ is unique up to a constant.


\subsection{Proof of Item (2)}

\begin{lemma}\label{Philip}
For $x$ and $y\in\mathbb T$, there is $K_1>0$ such that
\[d(\Phi(t,T,x),\Phi(t,T,y))\leq K_1d(x,y),\quad t\in[T-\tau,T].\]
\end{lemma}
\begin{proof}
Since the argument is local, we assume $\mathbb T$ is an open subset of $\mathbb R$. Then the distance $d(x,y)=|x-y|$. By Lemma \ref{uC2} (1), the lift $\bar\gamma:\mathbb R\to\mathbb R$ is monotone. Then for $x\geq y$, we have $\Phi(t,T,x)-\Phi(t,T,y)\geq 0$. By Lemma \ref{uC2} (2), both $\|Du_0\|_\infty$ and $\|D^2u_0\|_\infty$ is bounded. Therefore, there is $K_2>0$ large enough such that
\begin{align*}
&\frac{d}{dt}(\Phi(t,T,x)-\Phi(t,T,y))
\\&=\partial_p H(\Phi(t,T,x),Du_0(\Phi(t,T,x)))-\partial_pH(\Phi(t,T,y),Du_0(\Phi(t,T,y)))
\\ &\leq K_2(\Phi(t,T,x)-\Phi(t,T,y)),
\end{align*}
where $\Phi(T,T,x)-\Phi(T,T,y)=x-y$. 
By the comparison property of ordinary differential equations, we have
\[0\leq \Phi(t,T,x)-\Phi(t,T,y)\leq e^{\tau K_2}d(x,y)\] for $x\geq y$ and $t\in [T-\tau,T]$.
\end{proof}
For two different final condition $m^1_T$ and $m^2_T$, let $\bar m_i(t)=\Phi(t,T,\cdot)_{\#}m^i_T$, $i=1,2$. For $t\in [T-\tau,T]$, we have
\[d_1(\bar m_1(t),\bar m_2(t))=\sup_{\phi}\int_{\mathbb T}\phi(\Phi(t,T,x))d(m^1_T-m^2_T)\leq K_1d_1(m^1_T,m^2_T),\]
since $x\mapsto \phi(\Phi(t,T,x))$ is $K_1$-Lipschitz continuous by Lemma \ref{Philip}. By (\ref{cmt}) and (F),
\begin{align*}
|c(m^1_T)-c(m^2_T)|&=\frac{1}{\tau}\int_0^\tau |F(\bar m_1(s))-F(\bar m_2(s))|ds
\\ &=\frac{1}{\tau}\int_{T-\tau}^T |F(\bar m_1(s))-F(\bar m_2(s))|ds
\\ &\leq \frac{1}{\tau}\int_{T-\tau}^T Cd_1(\bar m_1(s),\bar m_2(s))ds
\leq CK_1d_1(m^1_T,m^2_T),
\end{align*}
which implies that $m_T\mapsto c(m_T)$ is Lipschitz continuous in $d_1$.


\subsection{Proof of Item (3)}

Let $y=\Phi(t,T,x)$, define the inverse $x=\Phi^{-1}(t,T,y)$.
\begin{lemma}
The map $(t,T,y)\mapsto \Phi^{-1}(t,T,y)$ is of class $C^1$.
\end{lemma}
\begin{proof}
We prove this lemma on the cover $\mathbb R$ of $\mathbb T$. Since
\[\frac{dx}{dt}=\frac{\partial H}{\partial p}(x,Du_0(x)),\]
solving the above ordinary differential equation directly, we have
\begin{equation}\label{NL}
  \int_y^{\Phi^{-1}(t,T,y)}(\partial_p H(x,Du_0(x)))^{-1}dx=T-t.
\end{equation}
Let
\[G(x):=\int(\partial_p H(x,Du_0(x)))^{-1}dx.\]
By Lemma \ref{uC2}, $Du_0$ is of class $C^1$. By (\ref{neq0}),
\begin{equation*}\label{v(x)}
  v(x):=\frac{\partial H}{\partial p}(x,Du_0(x))\neq 0,
\end{equation*}
and $v(x)$ is of class $C^1$. Then $G'(x)=v(x)^{-1}\neq 0$, which implies $G$ is invertible. We also have $G''(x)=-v'(x)v(x)^{-2}$, which implies that $G(x)$ is of class $C^2$. Then it has a $C^2$ inverse $G^{-1}(x)$. By (\ref{NL}) we have
\[G(\Phi^{-1}(t,T,y))-G(y)=T-t,\]
which implies that $\Phi^{-1}(t,T,y)=G^{-1}(G(y)+T-t)$ is of class $C^2$.
\end{proof}
When $m_T$ is $C^1$, by definition we have
\begin{align*}
\int_{\mathbb T}f(x)d\bar m(t)&=\int_{\mathbb T}f(x)d\Phi(t,T,\cdot)_{\#}m_T
=\int_{\mathbb T}f(\Phi(t,T,x))m_T(x)dx
\\ &=\int_{\mathbb T}f(y)m_T(\Phi^{-1}(t,T,y))d\Phi^{-1}(t,T,y)
\\ &=\int_{\mathbb T}f(y)m_T(\Phi^{-1}(t,T,y))\partial_y\Phi^{-1}(t,T,y)dy,\quad \forall f\in C(\mathbb T),
\end{align*}
where $y=\Phi(t,T,x)$. Then we finally get \[\bar m(x,t)=m_T(\Phi^{-1}(t,T,x))\partial_x\Phi^{-1}(t,T,x),\] which is of class $C^1$.

\subsection{Proof of Item (4)}

The existence and uniqueness of the weak solution of (\ref{ME1})-(\ref{ME3}) have been proved in Lemma \ref{unique}. Now we consider the large time behavior of the unique solution of (\ref{ME1})-(\ref{ME3}) for each $\varphi\in C(M)$, where $c=c(m_T)$. According to \cite[Theorem 6.4.9]{CS}, for almost every $x\in\mathbb T$, \[(x(s),p(s)):=(\Psi(s,T,x),Dw(\Psi(s,T,x),s)),\quad s\in(0,T)\] satisfies the Hamilton equation (\ref{Heq}) with $(x(T),p(T))=(x,Dw(x,T))$. By Lemma \ref{equi} and the continuous dependence of solutions of ODEs on the final values, for almost all $x$ and each $t\in[0,+\infty)$, we have $\Psi(s,T,x)\to \Phi(s,T,x)$ uniformly for $s\in[T-t,T]$ as $T\to+\infty$. Define
\[f_T(x)=\max_{s\in[T-t,T]}d(\Psi(s,T,x),\Phi(s,T,x)),\]
then $f_T(x)\to 0$ a.e. on $\mathbb T$.

For $m(s)$, we have
\begin{align*}
d_1(m(s),\bar m(s))&=\sup_{\phi}\int_{\mathbb T}(\phi(\Psi(s,T,x))-\phi(\Phi(s,T,x)))dm_T(x)
\\ &\leq \int_{\mathbb T}d(\Psi(s,T,x),\Phi(s,T,x))m_T(x)dx
\\ &\leq \int_{\mathbb T}f_T(x)m_T(x)dx, \quad s\in[T-\tau,T].
\end{align*}
Since $m_T(x)$ is bounded and $\mathbb T$ is compact, by the dominated convergence theorem, \[d_1(m(s),\bar m(s))\to 0,\quad \textrm{as}\quad T\to+\infty,\] uniformly for $s\in [T-t,T]$.

For $u(x,s)$, we first take suitable $u_0$. Then we have
\begin{align*}
&u(x,s)-\int_0^{T-t} F(m(\xi))d\xi-\bigg(\bar u(x,s)-\int_0^{T-t} F(\bar m(\xi))d\xi\bigg)
\\ &=w(x,s)+c_0s-u_0(x)+\int_{T-t}^s(F(m(\xi))-F(\bar m(\xi)))d\xi,
\end{align*}
where $w(x,s)+c_0s$ uniformly converges to $u_0$ as $T\to +\infty$. By (F) we have
\begin{align*}
&\bigg|\int_{T-t}^s(F(m(\xi))-F(\bar m(\xi)))d\xi\bigg| \leq C\int_{T-\tau}^Td_1(m(s),\bar m(s))ds
\\ &\leq C\int_{T-\tau}^T\int_{\mathbb T}d(\Psi(s,T,x),\Phi(s,T,x))m_T(x)dxds
\\ &=C\int_0^\tau\int_{\mathbb T}d(\Psi(T-\tau+s,T,x),\Phi(T-\tau+s,T,x))m_T(x)dxds
\\ &\leq C\int_0^\tau\int_{\mathbb T}f_T(x)m_T(x)dxds\to 0,\quad \textrm{as}\quad T\to+\infty,
\end{align*}
uniformly for $s\in [T-t,T]$ by the dominated convergence theorem.

\medskip

\section*{Acknowledgements}

The author 
would like to thank Professor J. Yan for many helpful discussions.

\end{document}